\documentclass[english,twoside]{article}
\usepackage{cite}
\usepackage{amsmath,amssymb}
\usepackage{latexsym}
\usepackage[english]{babel}
\usepackage[latin1]{inputenc}
\usepackage{amsthm}
\usepackage{graphicx}

\newcommand{\N}{\mathbb{N}}
\newcommand{\Z}{\mathbb{Z}}

\newcommand{\R}{\mathbb{R}}

\newcommand{\A}{\mathbb{A}}

\newcommand{\s}{\mathbb{S}}

\newcommand{\abs}[1]{\left|#1\right|}

\newcommand{\fonc}[3]{#1:#2\to#3}

\DeclareMathOperator{\rot}{Rot}

\DeclareMathOperator{\fix}{Fix}

\DeclareMathOperator{\supp}{Supp}

\DeclareMathOperator{\homeo}{Homeo}
\DeclareMathOperator{\dif}{Diff}

\newcommand{\homeou}[1]{\homeo_0\left(#1\right)}

\newcommand{\diff}[2]{\dif^{#1}_0\left(#2\right)}

\theoremstyle{theorem}
\newtheorem{leem}{Lemma}[section]
\newtheorem{coro}[leem]{Corollary}
\newtheorem{prop}[leem]{Proposition}
\newtheorem{theo}[leem]{Theorem}

\newtheorem*{theoa}{Theorem A}

\newtheorem*{theob}{Theorem B}

\theoremstyle{remark}
\newtheorem*{example}{Example}
\newtheorem{rema}{Remark}

\title{Distortion elements in group of diffeomorphisms of the 2-sphere}
\author{\textsc{Jonathan Conejeros}} 
\date{}

\begin{document}

\maketitle

\abstract{We prove that every distortion element in the group of diffeomorphisms of the 2-sphere which has some recurrent point that is not fixed is an irrational pseudo-rotation. Moreover we prove that the differential of a distortion element in the group of diffeomorphisms of the 2-sphere having at least three fixed points at a fixed point has a unique eigenvalue which is 1.}

\section{Introduction}

We recall the concept of distortion in a group. We will say that $f$ in a group $G$ is a distortion element in $G$ if it has
infinite order and there exists a finite subset $S$ of $G$ such that:
\begin{itemize}
  \item[(i)] $f$ belongs to the group generated by $S$,
  \item[(ii)]  If $\abs{f^n}$ is the word length of $f^n$ in the generators of $S$, then
  $$ \lim_{n\to\infty} \frac{\abs{f^n}}{n}=0.  $$
\end{itemize}

This limit exists because the sequence  $(\abs{f^n})_{n\in\N}$ is sub-additive. The simplest examples of groups which contain a distortion element of infinite order are the solvable Baumslag-Solitar groups:
$$BS(1; p) :=< f, g \, : \, gfg^{-1}= f^p >, \text{ where } p \geq 2. $$

Indeed, for the group relation we have for every integer $n \geq 1$: $g^nfg^{-n}= f^{p^n}$. Taking as generators the set $S = \{f, g\}$, we have
$ \abs{f^{p^n}}\leq 2n+1 $ and the element $f$ is distorted in the group $BS(1; p)$, with $p \geq2$. Another example is the discrete Heisenberg group $H_3$. It is the group of integer matrices of the form
$$\begin{pmatrix}
  1 & a & b \\
  0 & 1 & c \\
  0 & 0 & 1
\end{pmatrix}$$

There are generators $g$ and $h$ of $H_3$ such that their commutator $f=[g,h]$ has infinite order and generates the center of $H_3$. This implies that $f^{n^2}=[g^n,h^n]$ and hence that $f$ is a distortion element in $H_3$.

Before stating our main result we recall some results of distortion elements in subgroups of the group of homeomorphisms on manifolds. Given a manifold $M$, we will denote by $\diff{r}{M}$ (resp. $\homeou{M}$) the connected component of the group of $C^r$-diffeomorphisms (resp. homeomorphisms) of $M$ containing the identity. The following result due to Calegari and Freedman (see \cite{calegari-freedman}) shows that every homeomorphism of the $d$-dimensional sphere is distorted.

\begin{theo}[Calegari and Freedman, \cite{calegari-freedman}]\label{theorem every homeo distortion}
  Let $d \geq 1$. Then every element in the group $\homeou{\s^d}$ is distorted.
\end{theo}

 In this paper we will consider the group of diffeomorphisms of the 2-sphere. As it is cleverly noticed in \cite{calegari-freedman}, we have the following restrictions for a distortion element $f$ in $\diff{1}{M}$.

\begin{itemize}
  \item[($P_1$)] If $x$ is a fixed point of $f$, then eigenvalues of $Df(x)$ have absolue value $1$, and
  \item[($P_2$)] $f$ has zero topological entropy
\end{itemize}
The authors of \cite{calegari-freedman} use Oseledec's theorem and Pesin-Ruelle inequality to proving property 2. We note that recently Navas gives a most elementary proof of property 2 which is valid for bi-Lipschitz homeomorphisms (see \cite{navas}). The following result improves the above second property $(P_2)$.

\begin{theo}[Franks and Handel, \cite{franks-handel}]\label{theorem franks handel} Let $M$ be a closed oriented surface of genus at least two. Then there are no distortion element in $\diff{}{M,\mu}$ the subgroup of the group $\diff{1}{M}$ which preserve a Borel probability measure on $M$ with total support.
\end{theo}

Also we can find infinitely many distortion elements in $\diff{1}{M}$ for any manifold $M$. More precisely we can see the solvable Baumslag-Solitar
group (defined above) $BS(1; p)$, with $p \geq 2$ as a subgroup of $\diff{1}{M}$.

\begin{example}[Solvable Baumslag-Solitar group embeds in $\diff{1}{M}$]
Choose a ball B in a manifold M. Let us consider a $C^1$-action of $BS(1; p)$ on $[0,1]$ such that the generators $f$, $g$ satisfy $f'(0)=g'(0)=1$ and $f'(1)=g'(1)=1$  (see \cite{bo-mon-navas-rivas}). Extend it to an action on $B$ which fixes an interior point and the boundary $\partial B$ so that it can be extend as the identity outside $B$.
\end{example}

This example being an action of the line has no recurrence. Militon gives another class of distorted element in $\diff{\infty}{M}$. He uses Avila's techniques (see \cite{avila}) and a local perfection result (see \cite{haller}), he obtains that every recurrent element in $\diff{\infty}{M}$, i.e. for which arbitrarily large iterates are arbitrarily close to the identity in the $C^{\infty}$-topology, is distorted.

\begin{theo}[Militon,\cite{militon}]
Let $M$ be a compact manifold. Then every recurrent element in the group $\diff{\infty}{M}$ is distorted.
\end{theo}

For instance, irrational rotations of the 2-sphere are distorted (previously showed by Calegari and Freedman (see \cite{calegari-freedman})). More generally, using the Anosov-Katok method (see \cite{herman}  and \cite{fathi-herman}), we can build recurrent elements in the case of the 2-sphere which are not conjugate to a rotation (such an element are called irrational pseudo-rotation). We start by introducing some definitions. Let $f$ be a homeomorphism of $\s^2$. We say that $z\in \s^2$ is a (positively) recurrent point that is not fixed for $f$ if there are arbitrarily large (positively) iterates $f^n(z)$ which are arbitrarily close to $z$ and $f(z)\neq  z$. We say that a homeomorphism $f$ of $\s^2$ is an irrational pseudo-rotation if:
\begin{itemize}
  \item[(i)] $f$ has exactly two periodic points, $z_0$ and $z_1$, which are both fixed;
  \item[(ii)] there exists a point $z$ whose $\alpha$-limit set or $\omega$-limit set is not included in $\{z_0, z_1\}$; and
  \item[(iii)] if $\check{f}$ is a lift of $f|_{\s^2\setminus\{z_0,z_1\}}$ to the universal covering space of $\s^2\setminus\{z_0,z_1\}$, its unique rotation number is an irrational number.
\end{itemize}

Our main result is the following.

\begin{theoa}
  Let $f$ be a distortion element in $\diff{1}{\s^2}$ such that has a (positively) recurrent point that is not fixed. Then $f$ is an irrational pseudo-rotation.
\end{theoa}

Clearly, by Theorem \ref{theorem every homeo distortion} the hypothesis of differentiability is necessary. Moreover the solvable Baumslag-Solitar group embeds in $\diff{1}{\s^2}$ (above example) and Heisenberg group $H_3$ acts smoothly on $\s^2$ by restrinting and projectivizing the standard action of $GL(3,\R)$ on $\R^3$. These examples show that hypothesis of the existence of recurrence in Theorem A is also necessary.

What about above property 1 $(P_1)$ ? A priori this property can be improve with an additional very strong hypothesis. If a distortion element $f$ belongs to a finitely generated group which acts with a global fixed point $z$. Then the differential at $z$ induces an action of the group on $GL(2,\R)$, and as it is noticed in \cite{calegari-freedman} the eigenvalues of $Df(z)$ are roots of unity.

The property ``eigenvalues are roots of unity'' is again true in some cases for our examples of groups having distortion elements without any additional hypothesis.

Case of solvable Baumslag-Solitar group. Here Guelman and Liousse showed the existence of a $g$-minimal set $\Lambda$ contained in $\fix(f)$ (see \cite{guelman-liousse 1}). Using the fact that every point in $\Lambda$ is recurrent, the authors of \cite{guelman-liousse 1} proved, in \cite{guelman-liousse 2}, that for every $z\in \Lambda$ the eigenvalues of $Df(z)$ are roots of unity.

Case of Heisenberg group. If $H_3$ acts on $\s^2$ by diffeomorphisms, then Rib\'on insures the existence of an orbit of cardinality at most $2$ (see \cite{ribon}). Considering power of the generating elements (if necesary) we can suppose that $H_3$ acts with a global fixed point $z$. Then Franks and Handel (see \cite{franks-handel}) proved that $Df(z)=\pm Id$.

Our second result shows that for distortion elements in $\diff{1}{\s^2}$ any additional hypothesis is necessary.

\begin{theob}
  Let $f$ be a distortion element in $\diff{1}{\s^2}$ which has at least three fixed points. If $z$ is a fixed point of $f$, then $Df(z)$ has a unique eigenvalue which is 1.
\end{theob}

These results are part of a bigger plan which seeks to characterize dynamics of distortion elements of the $2$-sphere.

We will proved Theorem A using the new criterion for the existence of topological horseshoes for surface homeomorphisms which are isotopic to the identity recently find for Le Calvez and Tal (see \cite{lecalvez-tal}).

\subsection*{Acknowledgements}

The preparation of this article was funded by the FONDECYT. I would like to thanks Patrice Le Calvez and Fr\'ed\'eric Le Roux for suggesting some questions which motivated Theorem A. I would also like thanks Andr\'es Navas for some useful comments and remarks.

\section{Preliminary results}

In this section we recall some results of distortion elements in the group of diffeomorphisms of the $2$-sphere and homeomorphisms of the $2$-sphere with no topological horseshoes that we will use in the rest of the article.

\subsection{Distortion elements in the group of diffeomorphisms of the $2$-sphere}

The following result is Theorem \ref{theorem franks handel} for the case of the $2$-sphere.

\begin{theo}[Theorem 1.3, \cite{franks-handel}]
Let $f$ be a distortion element in $\diff{1}{\s^2}$ and let $\mu$ be an $f$-invariant Borel probability measure. If $f^n$ has at least three fixed points for some smaller integer $n$, then $\supp(\mu) \subset \fix(f^n)$.
\end{theo}

We state the following corollary which will be used in the rest of the article.

\begin{coro}\label{corollary franks handel theorem}
Let $f$ be a distortion element in $\diff{1}{\s^2}$. Suppose that $f$ has at least three fixed points, then
\begin{itemize}
  \item[(i)] the set of periodic points of $f$ coincides with that of its fixed points;
  \item[(ii)] if $X$ is a closed $f$-invariant set, then $X$ contains a fixed point; and
  \item[(iii)] if $h(f)$ denotes the topological entropy of $f$, then $h(f)=0$.
\end{itemize}
\end{coro}

\subsection{Homeomorphisms of the 2-sphere with no topological horseshoes}

In this section we state some results which were proved in \cite{lecalvez-tal}, where the authors study some properties of homeomorphisms of the 2-sphere with no topological horseshoes. A compact subset $Y$ of a manifold $M$ is a topological horseshoe if it is invariant by an iterated $f^q$ of $f$, $f^q\vert_Y$ is semi-conjugated to $\fonc{g}{Z}{Z}$ on a Hausdorff compact space, so that the fibers by the factor map are all finite with an uniform bound $m$ in their cardinality, and $g$ is semi-conjugated to the Bernoulli shift $\fonc{\sigma}{\{1,\cdots,N\}}{\{1,\cdots,N\}}$, where $N\geq 2$ is such that the preimage of every $s$-periodic sequence of $\{1,\cdots,N\}$ by the factor map contains a $s$-periodic point of $g$.

\begin{rema}\label{remark 1}
  Note that, if $h(f)$ denotes the topological entropy of $f$, then
 $$ qh(f)=h(f^q)\geq h(f^q\vert_Y)=h(g)\geq h(\sigma) =\log(N), $$
 and that $f^{rn}$ has at least $q^n/m$ fixed points for every $n\geq 1$.
\end{rema}

The first result is about rotation numbers for annular homeomorphisms with no topological horseshoes. We will write $\A= (\R/\Z)\times \R$ the open annulus
and denote by $\fonc{\pi}{\R^2}{\A}$ the universal covering projection of $\A$ and by $\fonc{\pi_1}{\R^2}{\R}$ the projection in the first coordinate.

\begin{theo}[Theorem A, \cite{lecalvez-tal}]\label{theorem A lecalves tal}
Let $f$ be a homeomorphism of $\A$ which is isotopic to the identity and let $\check{f}$ be a lift of $f$ to $\R^2$. We suppose that $f$ has no topological horseshoe. Then each point $z$ such that its $\omega$-limit set is non empty has a well defined rotation number $\rot_{\check{f}}(z)$, i.e. for every compact set $ K \subset \A$ and every increasing sequence of integers $(n_k)_{k\in\N}$ such that $z$ and $f^{n_k}(z)$ belong to $K$, we have
 $$  \lim_{k\to \infty} \frac{1}{n_k}(\pi_1(\check{f}^{n_k}(\check{z})-\check{z}  )  )= \rot_{\check{f}}(z) ,$$
where $\check{z}$ is a lift of $z$.
\end{theo}

\begin{theo}[Theorem F, \cite{lecalvez-tal}]\label{theorem F lecalves tal}
  Let $f$ be a homeomorphism of $\A$ which is isotopic to the identity and let $\check{f}$ be a lift of f to $\R^2$. We suppose that $\check{f}$ is fixed point free and that there exists a positively recurrent point $z$ such that $\rot_{\check{f}}(z)$ is well defined and equal to $0$. Then $f$ has a topological horseshoe.
\end{theo}

\begin{theo}[Theorem G, \cite{lecalvez-tal}]\label{theorem G lecalves tal}
  Let $f$ be an orientation-preserving homeomorphism of $\s^2$ that has no topological horseshoe. Then every non-wandering point $z$ whose $\alpha$-limit set or $\omega$-limit set in not included in the set of fixed points of $f$ lies in an open annulus $A$ which is a fixed point free and $f\vert_{A}$ is isotopic to the identity.

\end{theo}

\section{Main Proposition}

In this section, we will prove the following proposition which is key in the demonstrations of Theorems A and B.

\begin{prop}\label{main proposition}
 Let $f$ be a distortion element in $\diff{1}{\s^2}$ and let $z_0$ and $z_1$ be two distinct fixed points of $f$. Let $f_{ann}$ be the homeomorphism of the closed annulus $\overline{A}$ obtained from $\s^2$ by blowing up both points $z_0$ and $z_1$. If $\check{f}_{ann}$ is a lift of $f_{ann}$ to the universal covering space of $\overline{A}$, then it has a unique rotation number, i.e. there exists $\rho$ such that

 $$  \lim_{n\to \infty} \frac{1}{n}(\pi_1(\check{f}_{ann}^{n}(\check{z})-\check{z}  )  )= \rho,$$
where $\check{z}$ is a lift of $z\in \overline{A}$.

\end{prop}

To prove this proposition we need the concept of spread introduced by Franks and Handel in \cite{franks-handel}.\\

\textit{Spread.-} Let $\gamma$ be a smooth curve with endpoints $z_0$ and $z_1$ (smooth at the endpoints) and $\beta$ be a simple nullhomotopic closed curve on $\s^1$ such that the point $z_1$ is contained in the disk bounded by $\beta$. For any curve $\alpha$, the spread $L_{\gamma,\beta}(\alpha)$ is going to measure how many times $\alpha$ rotates around $\beta$ with respect to $\gamma$. More formally the definition of spread is the following: let $\overline{A}$ be the compact annulus obtained from $\s^2$ by blowing up both points $z_0$ and $z_1$. In this case the universal covering space of $\overline{A}$ can be identified with $\R\times [0,1]$ and the covering translation can be identified by $T(x,y)=(x+1,y)$. For each lift $\check{\alpha}$ of $\alpha$ and $\check{\gamma}$ of $\gamma$ to $\R\times [0,1]$, there exist integers $a<b$ such that $\check{\alpha} \cap T^{i}(\check{\gamma})\neq \emptyset $ if and only if $a<i<b$. Define
 $$  \check{L}_{\check{\beta},\check{\gamma}}(\check{\alpha})= \max\{0, b-a-2\}.$$
and

 $$  L_{\beta,\gamma}(\alpha)= \max_{\check{\alpha}}\{\check{L}_{\check{\beta},\check{\gamma}}(\check{\alpha})\}.$$

Define the spread of $\alpha$ with respect to $f$, $\beta$ and $\gamma$ as

$$  \sigma_{f,\beta,\gamma}(\alpha):= \liminf_{n\to \infty} \frac{L_{\beta,\gamma}(f^n(\alpha))}{n}. $$

In \cite{franks-handel} the authors give some criteria for undistortion, one of them is positive spread.

\begin{leem}
  Let $g_i$, $1\leq i \leq k $ be a finite set of elements of $\diff{1}{\s^2}$. There exists a constant $C > 0$ such that the following property holds: If $f$ belongs to the group generated by the $g_i$ and if $\abs{f^n}$ is the word length of $f^n$ in the generators $g_i$, then for all curves $\alpha$, $\beta$, $\gamma$ and all integer $n>0$,
$$  L_{\beta,\gamma}(f^n(\alpha)) \leq L_{\beta,\gamma}(\alpha)+ C\abs{f^n}. $$
\end{leem}

It follows the following proposition.

\begin{prop}\label{distortion imply spread 0}
  Let $f$ be a distortion element in $\diff{1}{\s^2}$. Then for all curves $\alpha$, $\beta$, $\gamma$ we have
  $$ \sigma_{f,\beta,\gamma}(\alpha)=0.$$
\end{prop}
\begin{proof}
  Since $f$ is a distortion element in $\diff{1}{\s^2}$, there exists $g_i$, $1\leq i \leq k $ a finite set of elements of $\diff{1}{\s^2}$ such that $f$ belongs to the group generated by the $g_i$. Moreover if $\abs{f^n}$ is the word length of $f^n$ in the generators $g_i$,
  $$ \lim_{n\to\infty} \frac{\abs{f^n}}{n}=0.  $$

According to the definition of spread and above lemma we have
$$ \sigma_{f,\beta,\gamma}(\alpha)=\liminf_{n\to \infty} \frac{L_{\beta,\gamma}(f^n(\alpha))}{n} \leq \liminf_{n\to\infty} \frac{L_{\beta,\gamma}(\alpha)+ C\abs{f^n}}{n}=0.$$
\end{proof}

The following result is about rotation numbers for homeomorphisms of the closed annulus with no topological horseshoe.

\begin{leem}
  Let $f$ be a homeomorphism of $(\R/\Z)\times [0,1]$ which is isotopic to the identity and let $\check{f}$ be a lift of $f$ to $\R \times [0,1]$. We
suppose that $f$ has no topological horseshoe. Then each point $z \in (\R/\Z)\times [0,1]$ has a well defined rotation number
$$  \rot_{\check{f}}(z)= \lim_{n\to \infty} \frac{\pi_1(\check{f}^n(\check{z}))- \pi_1(\check{z})}{n},   $$
where $\check{z}$ is a lift to $z$.

\end{leem}
\begin{proof}
   Write $f(x,y)=(f_1(x, y), f_2(x, y))$, we can extend $f$ to a homeomorphism $f^{\ast}$ of the open annulus $\A$ such that $f^{\ast}(x,y)=(f_1(x,1),y)$ if $y\geq 1$ and $f^{\ast}(x,y)=(f_1(x,0),y)$ if $y\leq 0$. Note that $f^{\ast}$ has no horseshoe, because $f$ has no. As $ (\R/\Z)\times [0,1]$ is included in a compact set of $\A$, the result follows from Theorem \ref{theorem A lecalves tal}.
\end{proof}

\subsection{Proof of Proposition \ref{main proposition}}

  In this section, we prove Proposition \ref{main proposition}. Suppose by contradiction that there exist $z$ and $z'$ (in the interior of $\overline{A}$) such that $\rot_{\check{f}_{ann}}(z)=\rho$ and $\rot_{\check{f}_{ann}}(z')=\rho'$, with $\rho <\rho'$. Conjugation we can suppose that $\check{f}_{ann}$ is a homeomorphism of $\R\times [0,1]$ and a lift of $\gamma$ to $\R\times [0,1]$ is $\{0\}\times [0,1]$. Let $\alpha$ be a smooth curve joining $z$ and $z'$ and let $\check{\alpha}$ be a lift of $\alpha$ with endpoints $\check{z}$ and $\check{z'}$. For $\epsilon < (\rho'-\rho)/2$ and for every $n$ large enought we have
$$ \pi_1(\check{f}^n_{ann}(\check{z})) < n\epsilon +n\rho + \pi_1(\check{z})  $$
and
$$ -n\epsilon +n\rho' + \pi_1(\check{z'}) <  \pi_1(\check{f}^n_{ann}(\check{z'}))  $$

Therefore we have

$$ -2\epsilon n +(\rho'-\rho)n <  \pi_1(\check{f}^n_{ann}(\check{z'}))-\pi_1(\check{f}^n_{ann}(\check{z})) \leq L_{\beta,\gamma}(f^n(\alpha))+ 2.   $$
And so
 $$ 0< \sigma_{f,\beta,\gamma}(\alpha).  $$
This contradicts Proposition \ref{distortion imply spread 0}.

Suppose now that there exists $z$ (in the interior $\overline{A}$) such that $\rot_{\check{f}_{ann}}(z)=\rho$ and that the rotation number of the boundary component associated to $z_0$ is $\rho_0$, with $\rho<\rho_0$. We have two cases.

Case 1. The local rotation set at $z_0$ is empty.
In this case following \cite{leroux} we have one of the next three possibilities
\begin{itemize}
  \item[(i)] $z_0$ is a global attractor,
  \item[(ii)] $z_0$ is a global repeller,
  \item[(iii)] $f$ is locally conjugate to the application $z \mapsto e^{2\pi i p/q} z (1+z^{qr})$, for some integers $p$, $q$ and $r$.
\end{itemize}

If possibility (i) or (iii) is valid, then there exists a point $z'$ (in the interior of $\overline{A}$) whose $\omega$-limit set is contained in the boundary component of $\overline{A}$ associated to $z_0$ and so $\rot_{\check{f}_{ann}}(z')=\rho_0$. This contradicts the first part. If possibility (ii) is valid for both boundary components of $\overline{A}$ the result follows by the first part.

Case 2. The local rotation set at $z_0$ is not empty. In this case following \cite{conejeros}, we know that the rotation set of the open annulus is an interval. Take a rational number written in an irreducible way $p/q$ such that $\rho<p/q<\rho_0$. We have again three possibilities (see \cite{conejeros}). Again, there exists a point $z'$ (in the interior of $\overline{A}$) such that $\rho< p/q \leq \rot_{\check{f}_{ann}}(z')$. This is again a contradiction with the first part.

This proves the result.




\section{Proof of main results}

\subsection{Proof of Theorem A}

In this section, we prove Theorem $A$. Let $f$ be a distortion element in $\diff{1}{\s^2}$. We will consider two cases:

\textbf{Case 1:} f has at least three fixed points.\\
Let $z$ be a (positively) recurrent point of $f$ which is not fixed. By Theorem \ref{theorem G lecalves tal} there exists an open annulus A which is maximal fixed point free and that contains $z$. As $f\vert_A$ is isotopic to the identity, by Corollary \ref{corollary franks handel theorem}, we can fix two fixed points $z_0$ and $z_1$ one in each connected component of the complement of $A$. Let $z_2$ be a third fixed point and let $\check{f}$ be a lift of $f\vert_{\s^2\setminus \{z_0,z_1\}}$ to the universal covering space of $\s^2\setminus \{z_0,z_1\}$ that fixes a lift of $z_2$. Then by Proposition \ref{main proposition} it has a unique rotation number which must be $\{0\}$. In this case the positively recurrent point $z$ in $A$ has a well defined rotation number $\rot_{\check{f}}(z)$ and it must be equal to 0. From Theorem \ref{theorem F lecalves tal}, $f\vert_A$ has a topological horseshoe. This contradicts Corollary \ref{corollary franks handel theorem} (see Remark \ref{remark 1}).\\

\textbf{Case 2:} $f$ has exactly two fixed points.\\
Let $z_0$ and $z_1$ be the two fixed points of $f$ and let $\check{f}$ be a lift of $f\vert_{\s^2\setminus \{z_0,z_1\}}$ to the universal covering space of $\s^2\setminus \{z_0,z_1\}$. Then as above it has a unique rotation number. If this is rational as above we can show that $f$ has a topological horseshoe which contradicts Corollary \ref{corollary franks handel theorem} (see Remark \ref{remark 1}). Hence the unique rotation number of $\check{f}$ is irrational and so $f$ is an irrational pseudo-rotation.

\subsection{Proof of Theorem B}

In this section, we prove Theorem $B$. Let $f$ be a distortion element in $\diff{1}{\s^2}$. Suppose that $f$ has at least three fixed points.
Fix us three distinct fixed points $z_0$, $z_1$ and $z$ of $f$. Let $f_{ann}$ be the homeomorphism of the closed annulus $\overline{A}$ obtained from $\s^2$ by blowing up the fixed points $z_0$ and $z$. If $\check{f}_{ann}$ is a lift of $f_{ann}$ to the universal covering space of $\overline{A}$ which fixes a lift of $z_2$, then by Proposition \ref{main proposition} it has a unique rotation number which must be $\{0\}$. Therefore the rotation number of both boundary components of $\overline{A}$ is zero. Hence $Df(z)$ has a unique eigenvalue which is $1$.


Departamento de Matem\'atica y Ciencia de la Computaci\'on,\\
Universidad de Santiago de Chile.\\
Avenida Libertador Bernardo O''Higgins 3363,\\
Estaci\'on Central, Santiago, Chile.\\
e-mail: jonathan.conejeros@usach.cl

\end{document}